\documentclass[reqno]{amsart}

\usepackage{amsthm}
\usepackage{amssymb}
\usepackage{amsmath}
\usepackage{amsfonts}
\usepackage{amsrefs}
\usepackage{textcomp}
\usepackage[T1]{fontenc}
\usepackage[utf8x]{inputenc}
\usepackage{textcomp}
\usepackage{wasysym}
\usepackage{stmaryrd}
\usepackage{esint}
\usepackage[all]{xy}
\usepackage{graphicx}
\usepackage{bbm}
\usepackage{mathtools}
\usepackage{hyperref}
 \usepackage{color}

\newtheorem{theorem}{Theorem}[section]

\newtheorem{definition}[theorem]{Definition}

\newtheorem{remark}[theorem]{Remark}
\newtheorem{example}[theorem]{Example}
\newtheorem*{theorem*}{\it Theorem}

\mathtoolsset{showonlyrefs}
\numberwithin{equation}{section}

\def\Xint#1{\mathchoice
   {\XXint\displaystyle\textstyle{#1}}%
   {\XXint\textstyle\scriptstyle{#1}}%
   {\XXint\scriptstyle\scriptscriptstyle{#1}}%
   {\XXint\scriptscriptstyle\scriptscriptstyle{#1}}%
   \!\int}
\def\XXint#1#2#3{{\setbox0=\hbox{$#1{#2#3}{\int}$}
     \vcenter{\hbox{$#2#3$}}\kern-.5\wd0}}

\def\dashint{\Xint-}

\def\1{\raisebox{1pt}{\rm{$\chi$}}}

\newcommand{\twopartdef}[4]
{
\left\{
		\begin{array}{ll}
			#1 & #2 \\
			#3 & #4
		\end{array}
	\right.
}

\newcommand{\threepartdef}[6]
{
	\left\{
		\begin{array}{lll}
			#1 & #2 \\
			#3 & #4 \\
			#5 & #6
		\end{array}
	\right.
}

\usepackage{geometry}
\title{A metric counterpart of the Gu-Yung formula}

\author[S. Buccheri and W. G\'{o}rny]{Stefano Buccheri and Wojciech G\'{o}rny}
	
\address{S. Buccheri: Faculty of Mathematics, Universit\"at Wien, Oskar-Morgerstern-Platz 1, 1090 Vienna, Austria
\hfill\break\indent
{\tt stefano.buccheri@univie.ac.at }}

\address{W. G\'{o}rny: Faculty of Mathematics, Universit\"at Wien, Oskar-Morgerstern-Platz 1, 1090 Vienna, Austria; Faculty of Mathematics, Informatics and Mechanics, University of Warsaw, Banacha 2, 02-097 Warsaw, Poland
\hfill\break\indent
{\tt  wojciech.gorny@univie.ac.at }
}


\date{\today}

\begin{document}
\maketitle
\begin{abstract}
In this note we consider a generalisation to the metric setting of the recent work~\cite{GY}. In particular, we show that under relatively weak conditions on a metric measure space $(X,d,\nu)$, it holds true that
\[
\bigg[ \frac{u(x)-u(y)}{d(x,y)^{\frac{s}{p}}} \bigg]_{L^p_w(X \times X, \nu \otimes \nu)} \approx  \| u \|_{L^p(X,\nu)},
\]
where $s$ is a generalised dimension associated to $X$ and $[\cdot]_{L^p_w}$ is the weak Lebesgue norm. We provide some counterexamples which show that our assumptions are optimal. 
\end{abstract}
\section{Introduction}
Given $p\ge1$ and $\alpha\in(0,1)$, let us start by recalling the well-known definition of Gagliardo seminorm of a function $u:\mathbb{R}^N\to \mathbb{R}$ in the Euclidean setting
\begin{equation}\label{gagliardo}
[u]_{p,\alpha}= \int_{\mathbb{R}^N} \int_{\mathbb{R}^N} \frac{|u(x) - u(y)|^p}{|x-y|^{N + \alpha p}} \, dx \, dy=\left\|\frac{u(x) - u(y)}{|x-y|^{\frac Np + \alpha}}\right\|_{L^p(\mathbb{R}^N\times\mathbb{R}^N)}^p.
\end{equation}
Roughly speaking, if $[u]_{p,\alpha}<\infty$ we can say that the derivatives of $u$ of order $\alpha$ are $L^p$-integrable. However, this heuristic interpretation fails at the endpoint values $\alpha=1$ and $\alpha=0$. Indeed, on one side, $[u]_{p,1}\neq \|\nabla u\|_p$ and actually $[u]_{p,1}<\infty$ if and only if $u$ is constant (see the seminal paper~\cite{brezis}). On the other side, $[u]_{p,0}\neq \|u\|_p$ and the energy $[u]_{p,0}$ is connected to the so-called logarithmic Laplacian (see for instance \cite{chen}).

A possible way to correct this mismatch is to consider suitable renomalisations of the Gagliardo seminorm. Concerning the case when $\alpha \to 1^-$, Bourgain, Brezis and Mironescu proved that (see \cite{BBM1} and~\cite{BBM2}) for $u\in W^{1,p}(\Omega)$ it holds that
\begin{equation}\label{BBM}
\lim_{\alpha \rightarrow 1^-} (1-\alpha) \int_{\Omega} \int_{\Omega} \frac{|u(x) - u(y)|^p}{|x-y|^{N + \alpha p}} \, dx \, dy = K_{p,N} \| \nabla u \|^p_{L^p(\Omega)},
\end{equation}
where $\Omega\subset\mathbb{R}^N$ is a bounded smooth domain, and
$$K_{p,N}=\int_{\mathbb{S}^{N-1}}|e\cdot\omega|^p \, d\omega,$$
where $e$ is any unit vector in $\mathbb{R}^N$. Furthermore, this result serves also as a characterisation of Sobolev spaces, in the sense that a function $u$ lies in $W^{1,p}(\Omega)$ if and only if the left-hand side of \eqref{BBM} is finite. Formula \eqref{BBM} has been source of inspiration for many other contributions; without the intention of being exhaustive, let us mention some classical extensions of this result in the Euclidean case due to D\'{a}vila~\cite{Davila}, Ponce~\cite{Ponce}, Nguyen~\cite{Nguyen} and Leoni-Spector~\cite{Leoni}, as well as the more recent papers \cite{BCCS, AB, DDP}.

On the other hand, to relate the fractional seminorm with the $L^p$ norm as $\alpha \rightarrow 0^+$, Maz'ya and Shaposhnikova shown in \cite{MS} that whenever
$$u\in\bigcup_{s\in(0,1)} W^{s,p}_0(\mathbb{R}^N),$$
where $W^{s,p}_0(\mathbb{R}^N)$ denotes the completion of smooth functions with compact support in the fractional Sobolev norm, it holds that
\begin{equation}\label{MazSha}
\lim_{\alpha \rightarrow 0^+} \alpha \int_{\mathbb{R}^N} \int_{\mathbb{R}^N} \frac{|u(x) - u(y)|^p}{|x-y|^{N + \alpha p}} \, dx \, dy = \frac{2\omega_N}{pN} \| u \|^p_{L^p(\mathbb{R}^N)},
\end{equation}
Formula \eqref{MazSha} has been later generalised for anisotropic norms on $\mathbb{R}^N$ in \cite{Lud}.

A further intriguing step in this line of investigation has been done in the recent paper \cite{BVY}, where the authors consider the weak $L^p$ norm of the difference quotient in \eqref{gagliardo}, namely
\[
\left[\frac{u(x) - u(y)}{|x-y|^{\frac Np + \alpha}}\right]_{L^p_w(\mathbb{R}^N\times\mathbb{R}^N)}= \sup_{\lambda > 0} \lambda^p \mathcal{L}^{2N}\left(\bigg\{(x,y)\in \mathbb{R}^N \times \mathbb{R}^N: \ \frac{|u(x)-u(y)|}{|x-y|^{\frac{N}{p}+\alpha}} \geq \lambda \bigg\}\right).
\]
The weak $L^p$ space $L^p_w$, also known in the literature as Marcinkiewicz space $M^p$ or the Lorentz space $L^{p,\infty}$, is a slightly larger space than the ordinary Lebesgue space $L^p$.
The main result of~\cite{BVY} shows that there exist two positive constants $c_1$ and $c_2$ depending only on the dimension such that for all $u\in C^{\infty}_c(\mathbb{R}^N)$
\[
c_1\|\nabla u\|_{L^p(\mathbb{R}^N)}\le \left[\frac{u(x) - u(y)}{|x-y|^{\frac Np + 1}}\right]_{L^p_w(\mathbb{R}^N\times\mathbb{R}^N)}\le c_2 \|\nabla u\|_{L^p(\mathbb{R}^N)}.
\]
and moreover
\[
\lim_{\lambda\to\infty} \lambda^p \mathcal{L}^{2N}\left(\bigg\{(x,y)\in \mathbb{R}^N \times \mathbb{R}^N: \ \frac{|u(x)-u(y)|}{|x-y|^{\frac{N}{p}+\alpha}} \geq \lambda \bigg\}\right) =\frac{K_{p,N}}{N}\|\nabla u\|_{L^p(\mathbb{R}^N)}^p,
\]
where
\[
K_{p,N}=\int_{\mathbb{S}^{N-1}}|e\cdot\omega|^p \, d\omega.
\]
Note that this is the same constant as in the classical Bourgain-Brezis-Mironescu formula. The previous formulas were later extended to the case $u\in W^{1,p}(\mathbb{R}^N)$ in \cite{Poli}.

The counterpart of \eqref{MazSha} in this Marcinkiewicz setting arrived soon and in \cite{GY} it has been shown by Gu and Yung that the measure of the set
\begin{equation}
E_{\lambda}= \bigg\{(x,y)\in \mathbb{R}^N \times \mathbb{R}^N: \ \frac{|u(x)-u(y)|}{|x-y|^{\frac{N}{p}}} \geq \lambda \bigg\}
\end{equation}
is related to the $L^p$ norm of $u$ in the following way:
for all $u \in L^p(\mathbb{R}^N)$, we have
\begin{equation}\label{GuYung1}
2\omega_N \| u \|_{L^p(\mathbb{R}^N)}^p \leq \bigg[ \frac{u(x)-u(y)}{|x-y|^{\frac{N}{p}}} \bigg]^p_{L^p_w(\mathbb{R}^{N} \times \mathbb{R}^N)} \leq 2^{p+1}\omega_N \| u \|_{L^p(\mathbb{R}^N)}^p.
\end{equation}
The upper bound was first observed in \cite{DM}. Moreover, the lower bound can be improved in the following way:
\begin{equation}\label{GuYung2}
\lim_{\lambda \rightarrow 0^+} \lambda^{p} \mathcal{L}^{2N}(E_\lambda) = 2\omega_N \| u \|_{L^p(\mathbb{R}^N)}^p,
\end{equation}
where $\omega_N$ is the volume of the unit ball. Later, this result was generalised to anisotropic norms on $\mathbb{R}^N$ in \cite{GH}, a family of weights on $\mathbb{R}^N$ in \cite{BSSY} (see also \cite{PYYZ}), and to Orlicz modulars in place of the $L^p$ norms in \cite{KMS}.

The Bourgain-Brezis-Mironescu and Maz'ya-Shaposhnikova formulas have also inspired numerous works in the setting of metric measure spaces. A first result in this direction was an analogue of formula \eqref{BBM} in Carnot groups, which was proved in \cite{Bar}, relying strongly on their homogeneous structure. Subsequently, characterisations of Sobolev and BV spaces using formulas of type \eqref{BBM} which are valid up to a multiplicative constant were shown in \cite{dMS}, \cite{MMS} and \cite{Mun}. A version of the Bourgain-Brezis-Mironescu formula with an explicit constant was first obtained in \cite{Gor2022JGA} under a condition on the local geometry of the space, i.e., that for $\nu$-a.e. point the Gromov-Hausdorff tangent is a Euclidean space or the Heisenberg group. This result was subsequently generalised in \cite{HP2} to the case of more general mollifiers (see also \cite{LPZ}). Furthermore, some interesting dimensionless results of this type was shown for Carnot groups \cite{GT1,GT2} and RCD spaces \cite{BPP}, and a compactness result inspired by \cite[Theorem 4]{BBM1} in metric measure spaces was shown in the recent preprint \cite{ARB}. On the other hand, the Maz'ya-Shaposhnikova formula was generalised to the metric setting in \cite{HP}. The authors introduced a geometric condition on the metric measure space, called the asymptotic volume ratio, under which a counterpart of formula \eqref{MazSha} is valid for a general family of mollifiers. This line of research, this time for general metric measure spaces which satisfy a Bishop–Gromov type inequality, is continued in the very recent paper \cite{HPXZ}.

The aim of this note it to generalise formulas \eqref{GuYung1} and \eqref{GuYung2} due to Gu-Yung \cite{GY} to the metric setting and understand which assumptions are crucial to the proofs. As we shall see, the main ingredients that we are going to use are Ahlfors regularity and the asymptotic volume ratio (AVR) introduced by Han and Pinamonti in~\cite{HP} (see Section \ref{preliminaries} for the precise definitions). These types of assumptions provide the considered metric measure spaces with a notion of generalised dimension $s$. For instance, Ahlfors regularity reads as
\[
\nu(B(x,r)) \approx r^s
\]
where $B(x,r)$ is a ball of radius $r$ (with center $x \in X$) according to the considered metric and~$\nu$~is the measure the space is endowed with. This generalised dimension $s$ will play the role of the parameter $N$ of formula \eqref{GuYung1}, see Theorems \ref{thm:maintheorem1} and \ref{thm:maintheorem2} in Section \ref{sec:mainresults}. Looking carefully at the proofs of these results, one realizes that it is possible to consider also a more general interaction between the metric structure and the measure. More precisely, we also consider assumptions of the following type
\[
\nu(B(x,r)) \approx f(r),
\]
where $f$ is a convex increasing function with $f(0)=0$. As we shall see in Theorems \ref{thm:cor1} and \ref{thm:cor2}, it~is possible to generalise Gu-Yung formula to this context as well. We also provide some examples that show the optimality of the considered assumptions (see Examples \ref{ex:finitemeasure}, \ref{ex:counterexample}, and \ref{ex:nolimitaslambdagoestozero}).
The considered generality allows us to cover, for instance, Riemannian or sub-Riemannian manifolds (other examples can be found in Section \ref{sec:applications}).

\section{Preliminaries}\label{preliminaries}

Throughout this paper, by a {\it metric measure space} we mean a triple $(\mathbb{X}, d, \nu)$, where $(\mathbb{X}, d)$ is a complete and separable metric space, and $\nu$ is a nonzero non-negative Borel measure in $(\mathbb{X}, d)$ which is finite on bounded sets. Taking $p \geq 1$, our main object of interest is the weak-$L^p$ norm (also called the Marcinkiewicz norm). Given a measure space $(Z,\mu)$ and a $\mu$-measurable function $f: Z \rightarrow \mathbb{R}$, the weak-$L^p$ norm of $f$ (taken to power $p$ for convenience) is defined as
\begin{equation}
[f]^p_{L^p_w(Z,\mu)} = \sup_{\lambda > 0} \, \lambda^p \mu(\{ z \in Z: \, |f(z)| \geq \lambda \}).
\end{equation}
The weak-$L^p$ space is a Banach space which consists of measurable functions for which the above expression is finite, i.e.,
\begin{equation}
L^p_w(Z,\mu) = \{ f: Z \rightarrow \mathbb{R} \quad \mu\mbox{-measurable}: \,\, [f]_{L^p_w(Z,\mu)} < \infty \}.
\end{equation}
The setting to which we apply this definition is $Z = X \times X$ and $\mu = \nu \otimes \nu$. More precisely,
\begin{equation}
[f]^p_{L^p_w(X \times X, \nu \otimes \nu)} = \sup_{\lambda > 0} \, \lambda^p (\nu \otimes \nu)(\{ (x,y) \in X \times X: \, |f(x,y)| \geq \lambda \}).
\end{equation}
We are mainly interested in the study of functions of the form 
$$f(x,y) = \frac{u(x) - u(y)}{d(x,y)^{\frac{s}{p}}},$$
where $u: X \rightarrow \mathbb{R}$ is a measurable function and $s > 0$ plays the role of the dimension of the space. 

Let~us~now present several conditions on the metric measure space which imply existence of a generalised dimension, which will then be used in the statements of our main results in Section \ref{sec:mainresults}. We say that $\nu$ is {\it doubling} if there exists $C_d > 0$ such that, for all $x \in X$ and $r > 0$,
 we have
\begin{equation}
0 < \nu(B(x,2r)) \leq C_d \nu(B(x,r)).
\end{equation}
To a metric measure space with a doubling measure one can associate a homogeneous dimension $s = \log_2 C_d$ in the following way: every doubling measure $\nu$ is {\it Bishop-Gromov regular} of dimension $s$, i.e., there exists a constant $K \geq 1$ such that
\begin{equation}\label{bigr}
\nu(B(x,R)) \leq K \frac{R^s}{r^s} \nu(B(y,r)),
\end{equation}
for all $x,y \in X$ such that $d(x,y) \leq R$ and all $r \in (0,R)$. Moreover, every metric space $(X,d)$ equipped with a doubling measure $\nu$ is proper (i.e., bounded closed subsets are compact), the measure $\nu$ is supported on the whole space $X$, and it assigns zero measure to spheres.

A second condition we consider is Ahlfors regularity. We say that $\nu$ is {\it Ahlfors regular} if there exist $C_a, C_A > 0$ and $s > 0$ such that for all $x \in X$ and $r > 0$
\begin{equation}\label{ahlfreg}
C_a r^s \leq \nu(B(x,r)) \leq C_A r^s.
\end{equation}
It is a stronger assumption than the previous one and it is easy to see that it implies the doubling condition. This condition can also be used in several different variants: the measure $\nu$ is {\it upper Ahlfors regular}, if only the upper bound in \eqref{ahlfreg} holds, and {\it lower Ahlfors regular} if only the lower bound in \eqref{ahlfreg} holds. Furthermore, one may also consider a more restrictive setting, in which we require that $C_a = C_A$ and inequality \eqref{ahlfreg} is in fact an equality (as is the case for Euclidean spaces and Carnot groups).

Finally, we present a condition which is independent of the ones presented before, called the asymptotic volume ratio. It was introduced by Han and Pinamonti in \cite{HP} and concerns the asymptotic behaviour of the volume of large balls.

\begin{definition}\label{dfn:avr}
We say that the metric measure space $(X,d,\nu)$ admits an asymptotic volume ratio with dimension $s \in (0,\infty)$, if there exists a limit
\begin{equation}\label{avr}
{\rm AVR} = \lim_{r \rightarrow \infty} \frac{\nu(B(x_0,r))}{r^s} \in (0,+\infty)
\end{equation}
for some (equivalently: all) $x_0 \in X$. This definition does not depend on the choice of $x_0$ (see \cite{HP}).
\end{definition}

This condition was first introduced in order to study a metric analogue of the classical Maz'ya-Shaposhnikova formula in metric measure spaces, making it a natural assumption in our setting. Examples of spaces which satisfy some of the above assumptions, adapted to the settings of our main results (Theorems \ref{thm:maintheorem1} and \ref{thm:maintheorem2}), are given in Section \ref{sec:applications} together with a discussion on how to apply our results.

\section{Main results}\label{sec:mainresults}

In this section, we show our main results, which are Theorems \ref{thm:maintheorem1} and \ref{thm:maintheorem2}. In fact, they may be viewed as two variants of the same result, closely connected but with separate sets of assumptions which require some modification of the proofs. Later, we present how to generalise them to a more general setting in Theorems \ref{thm:cor1} and \ref{thm:cor2}. We comment on the optimality of the assumption at a later point, see Examples \ref{ex:finitemeasure}, \ref{ex:counterexample}, and \ref{ex:nolimitaslambdagoestozero}.

\begin{theorem}\label{thm:maintheorem1}
Suppose that $\nu(X) = + \infty$. Assume that $\nu$ is upper Ahlfors regular, namely there exists $C_A > 0$ such that
\begin{equation}\label{semiahlfor}
\nu(B(x,r)) \leq C_A r^s\ \ \ \forall \, x\in X, \, r>0.
\end{equation}
Furthermore, we require that its asymptotic value ratio ${\rm AVR}$ (see Definition \ref{dfn:avr}) corresponding to the exponent $s$ is finite, and that the measure of balls with a fixed center is continuous as a function of radius, i.e.,
\begin{equation}\label{eq:continuity}
\mbox{the map } r \mapsto \nu(B(x,r)) \mbox{ is continuous for all } x \in X.
\end{equation}
Then, there exist constants $c_1, c_2, c_3 > 0$ such that for all $p \in [1,\infty)$ and $u \in L^p(X,\nu)$ we have
\begin{equation}\label{eq:bounds}
c_1 \| u \|_{L^p(X,\nu)}^p \leq \bigg[ \frac{u(x)-u(y)}{d(x,y)^{\frac{s}{p}}} \bigg]^p_{L^p_w(X \times X, \nu \otimes \nu)} \leq c_2 \| u \|_{L^p(X,\nu)}^p.
\end{equation}
Moreover, the lower bound can be improved in the following way. If we denote
\begin{equation}
E_{\lambda}= \bigg\{(x,y)\in X \times X: \, \frac{|u(x)-u(y)|}{d(x,y)^{\frac{s}{p}}} \geq \lambda \bigg\}
\end{equation}
we have
\begin{equation}\label{eq:limit}
\lim_{\lambda \rightarrow 0^+} \lambda^{p} (\nu \otimes \nu) (E_\lambda) = c_3 \| u \|_{L^p(X,\nu)}^p.
\end{equation}
We may take $c_1 = c_3 = 2 \cdot {\rm AVR}$ and $c_2 = 2^{p+1} C_A$. 
\end{theorem}

Let us note that assumption \eqref{eq:continuity}, while a bit technical, is satisfied for instance when the measure $\nu$ is doubling. Moreover, as we will see in the proof, for the upper bound in \eqref{eq:bounds} only the upper Ahlfors regularity of $\nu$ is needed. The proof roughly follows the outline given in \cite{GY}.

\begin{proof}
\emph{Step 1.} First, we prove the second inequality in \eqref{eq:bounds}. Notice that when $(x,y) \in E_\lambda$, then either $|u(x)| \geq \frac12 \lambda d(x,y)^{\frac{s}{p}}$ or $|u(y)| \geq \frac12 \lambda d(x,y)^{\frac{s}{p}}$. So, by the Fubini theorem we have that
\begin{align}\label{eq:estimatefromabove1}
\lambda^p (\nu \otimes \nu)(E_\lambda) &\leq \lambda^p \int_X \int_X \1_{\{ (x,y): \, |u(x)| \geq \frac12 \lambda d(x,y)^{\frac{s}{p}} \} } \, d\nu(y) \, d\nu(x) \\
&\qquad\qquad\qquad\qquad\qquad + \lambda^p \int_X \int_X \1_{\{ (x,y): \, |u(y)| \geq \frac12 \lambda d(x,y)^{\frac{s}{p}} \} } \, d\nu(x) \, d\nu(y) \\
&= \lambda^p \int_X \int_X \1_{\{ (x,y): \, d(x,y) \leq (\frac{2}{\lambda} |u(x)|)^{\frac{p}{s}} \} } \, d\nu(y) \, d\nu(x) \\
&\qquad\qquad\qquad\qquad\qquad + \lambda^p \int_X \int_X \1_{\{ (x,y): \, d(x,y) \leq (\frac{2}{\lambda} |u(y)|)^{\frac{p}{s}} \} } \, d\nu(x) \, d\nu(y) =: {\rm I} + {\rm II}.
\end{align}
Each of these summands can be estimated  taking advantage of \eqref{semiahlfor}: we have
\begin{align}\label{eq:estimatefromabove2}
{\rm I} = \lambda^p\int_X \int_X &\1_{\{ (x,y): \, d(x,y) \leq (\frac{2}{\lambda} |u(x)|)^{\frac{p}{s}} \} } \, d\nu(y) \, d\nu(x) \\
&\qquad\qquad\qquad = \lambda^p \int_X \nu \bigg(B \bigg (x, \bigg(\frac{2}{\lambda} |u(x)| \bigg)^{\frac{p}{s}} \bigg) \bigg) \, d\nu(x) \leq 2^p C_A \int_X |u(x)|^p \, d\nu(x)
\end{align}
and
\begin{align}\label{eq:estimatefromabove3}
{\rm II} = \lambda^p\int_X \int_X &\1_{\{ (x,y): \, d(x,y) \leq (\frac{2}{\lambda} |u(y)|)^{\frac{p}{s}} \} } \, d\nu(y) \, d\nu(x) \\
&\qquad\qquad\qquad = \lambda^p \int_X \nu \bigg(B \bigg (y, \bigg(\frac{2}{\lambda} |u(y)| \bigg)^{\frac{p}{s}} \bigg) \bigg) \, d\nu(y) \leq 2^p C_A \int_X |u(y)|^p \, d\nu(y).
\end{align}
Collecting the estimates \eqref{eq:estimatefromabove1}, \eqref{eq:estimatefromabove2} and \eqref{eq:estimatefromabove3}, we get
\begin{align}\label{eq:estimatefromabove4}
\lambda^p (\nu \otimes \nu)(E_\lambda) \leq 2^p C_A \int_X |u(x)|^p \, d\nu(x) + 2^p C_A \int_X |u(y)|^p \, d\nu(y) = 2^{p+1} C_A \| u \|_{L^p(X,\nu)}^p.
\end{align}
Taking the supremum, we obtain the inequality from above in \eqref{eq:bounds} with the constant $c_2 = 2^{p+1} C_A$.

{\flushleft \it Step 2.} Now, we prove the first inequality in \eqref{eq:bounds}. This will be done by proving \eqref{eq:limit}, because the limit as $\lambda \rightarrow 0^+$ is clearly smaller or equal to the supremum. Similarly to the proof in \cite{GY}, we first consider the case when $u$ has bounded support.

Throughout the rest of this Step, assume that $u$ is boundedly supported. Fix $x_0 \in X$ and $R > 0$ such that $\mbox{supp } u \subset B(x_0,R)$. Denote
\begin{equation}
H_\lambda = E_\lambda \cap \{ (x,y) \in X \times X: \, d(x_0,y) > d(x_0,x) \}
\end{equation}
and
\begin{equation}
H'_\lambda = E_\lambda \cap \{ (x,y) \in X \times X: \, d(x_0,y) < d(x_0,x) \}.
\end{equation}
Notice that by symmetry we have $(\nu \otimes \nu)(H_\lambda) = (\nu \otimes \nu)(H'_\lambda)$. Also, the set
\begin{equation}
H''_\lambda = E_\lambda \cap \{ (x,y) \in X \times X: \, d(x_0,y) = d(x_0,x) \}
\end{equation}
has zero $(\nu \otimes \nu)$-measure under assumption \eqref{eq:continuity}: boundaries of balls have zero measure, so we integrate~$0$ using the Fubini theorem. Hence,
\begin{equation}\label{eq:twice}
(\nu \otimes \nu)(E_\lambda) = 2(\nu \otimes \nu)(H_\lambda),
\end{equation}
so it suffices to compute the limit with $H_\lambda$ in place of $E_\lambda$.

Now, notice that if $(x,y) \in H_\lambda$, then at least one of $x,y$ lies in the ball $B(x_0,R)$; but, since $d(x_0,x) < d(x_0,y)$, we necessarily have $x \in B(x_0,R)$. For such $x$, denote
\begin{equation}
H_{\lambda,x} := \bigg\{ y \in X: \, d(x_0,y) > d(x_0,x), \, \frac{|u(x)-u(y)|}{d(x,y)^{\frac{s}{p}}} \geq \lambda \bigg\}
\end{equation}
and
\begin{equation}
H_{\lambda,x,R} := H_{\lambda,x} \backslash B(x_0,r) = \bigg\{ y \in X: \, d(x_0,y) \geq R, \, d(x,y) \leq \bigg(\frac{|u(x)-u(y)|}{\lambda} \bigg)^{\frac{p}{s}} \bigg\},
\end{equation}
but since for $y \notin B(x_0,R)$ we have $u(y) = 0$, we have
\begin{equation}
H_{\lambda,x,R} = \bigg\{ y \in X: \, d(x_0,y) \geq R, \, d(x,y) \leq \bigg(\frac{|u(x)|}{\lambda} \bigg)^{\frac{p}{s}} \bigg\}.
\end{equation}
From the two above definitions, we immediately get
\begin{equation}\label{eq:inclusions}
H_{\lambda,x,R} \subset H_{\lambda,x} \subset H_{\lambda,x,R} \cup B(x_0,R).
\end{equation}
Using the Fubini theorem we get
\begin{equation}
(\nu \otimes \nu)(H_\lambda) = \int_{X} \nu(H_{\lambda,x}) \, d\nu(x) = \int_{B(x_0,R)} \nu(H_{\lambda,x}) \, d\nu(x).
\end{equation}
We need to estimate from above and below the measures of $H_{\lambda,x}$ for all $x \in B(x_0,R)$. We will do this using \eqref{eq:inclusions}. From the first inclusion, using \eqref{semiahlfor}, we have
\begin{equation}
\nu(H_{\lambda,x}) \geq \nu(H_{\lambda,x,R}) \geq \nu \bigg(B \bigg(x, \bigg(\frac{|u(x)|}{\lambda} \bigg)^{\frac{p}{s}} \bigg) \bigg) - \nu(B(x_0,R))\geq  \nu \bigg(B \bigg(x, \bigg(\frac{|u(x)|}{\lambda} \bigg)^{\frac{p}{s}} \bigg) \bigg) - C_A R^s.
\end{equation}
Multiplying the above inequalities by $\lambda^p$ and integrating it over $B(x_0,R)$ (which contains the support of $u$) we get
\begin{equation}
\lambda^p (\nu \otimes \nu)(H_\lambda)  \geq \lambda^p \int_X \nu \bigg( B \bigg(x, \bigg(\frac{|u(x)|}{\lambda}\bigg)^{\frac{p}{s}} \bigg) \bigg) \, d\nu(x)- \lambda^p C_A^2 R^{2s}.
\end{equation}
Now, we will take the limit of such an expression as $\lambda\to 0$. In order to do it, notice that the function 
$$f_{\lambda}(x) =\lambda^p \cdot \nu \bigg(B \bigg(x, \bigg(\frac{|u(x)|}{\lambda} \bigg)^{\frac{p}{s}} \bigg) \bigg)$$
satisfies
\begin{equation}
0\leq f_{\lambda}(x)\leq C_A |u(x)|^p \qquad \mbox{and} \qquad f_{\lambda}\to {\rm AVR} \cdot |u(x)|^p \quad \nu-\mbox{a.e} \ \mbox{in } X;
\end{equation}
indeed, by the upper Ahlfors regularity condition \eqref{semiahlfor}
\begin{equation}
f_\lambda(x) = \lambda^p \cdot \nu \bigg(B \bigg(x, \bigg(\frac{|u(x)|}{\lambda} \bigg)^{\frac{p}{s}} \bigg) \bigg) \leq \lambda^p C_A \bigg( \bigg(\frac{|u(x)|}{\lambda} \bigg)^{\frac{p}{s}} \bigg)^s = C_A |u(x)|^p
\end{equation}
and making the change of variables $r= (\frac{|u(x)|}{\lambda})^{\frac{p}{s}}$, by definition of the asymptotic volume ratio
\begin{equation}
\lim_{\lambda \rightarrow 0^+} f_\lambda(x) = \lim_{\lambda \rightarrow 0^+} \lambda^p \cdot \nu \bigg(B \bigg(x, \bigg(\frac{|u(x)|}{\lambda} \bigg)^{\frac{p}{s}} \bigg) \bigg) = \lim_{r \rightarrow \infty} \frac{|u(x)|^p}{r^s} \nu(B(x,r)) = {\rm AVR} \cdot |u(x)|^p.
\end{equation}
Then, the dominated convergence theorem implies that 
\begin{equation}\label{eq:estimatefrombelowavr}
\liminf_{\lambda \rightarrow 0^+}\lambda^p (\nu \otimes \nu)(H_\lambda) \geq {\rm AVR} \int_X|u(x)|^p \, d\nu(x).
\end{equation}
Similarly, from the second inclusion in \eqref{eq:inclusions}, we get that
\begin{equation}\label{eq:estimatefromabove}
\nu(H_{\lambda,x}) \leq \nu(H_{\lambda,x,R}) + \nu(B(x_0,R)) \leq \nu \bigg(B \bigg(x, \bigg(\frac{|u(x)|}{\lambda} \bigg)^{\frac{p}{s}} \bigg) \bigg) + C_A R^s,
\end{equation}
and that 
\begin{equation}
\limsup_{\lambda \rightarrow 0^+}\lambda^p (\nu \otimes \nu)(H_\lambda) \leq {\rm AVR} \int_X|u(x)|^p \, d\nu(x).
\end{equation}
Collecting estimates \eqref{eq:estimatefrombelowavr} and \eqref{eq:estimatefromabove}, taking into account \eqref{eq:twice}, we obtain that 
\begin{equation}
\lim_{\lambda \rightarrow 0^+} \lambda^p (\nu \otimes \nu)(E_\lambda) = 2 \cdot {\rm AVR} \int_X|u(x)|^p \, d\nu(x),
\end{equation}
and consequently \eqref{eq:limit} holds in the case when $u$ is boundedly supported.

{\it \flushleft Step 3.} Throughout this Step, we consider general functions $u$ which are not necessarily boundedly supported. We will rely on the estimate in the previous case and estimate the error from taking a boundedly supported approximation. Denote
\begin{equation}
u_R = u \cdot \1_{B(x_0,R)}
\end{equation}
and
\begin{equation}
v_R = u - u_R.
\end{equation}
Note that since $u \in L^p(X,\nu)$, both functions also lie in $L^p(X,\nu)$ and $v_R \rightarrow 0$ in $L^p(X,\nu)$ as $R \rightarrow \infty$. Now, fix $\sigma \in (0,1)$ and denote
\begin{equation}
A_1 = \bigg\{ (x,y) \in X \times X: \frac{|u_R(x) - u_R(y)|}{d(x,y)^{\frac{s}{p}}} \geq (1-\sigma) \lambda \bigg\}
\end{equation}
and
\begin{equation}
A_2 = \bigg\{ (x,y) \in X \times X: \frac{|v_R(x) - v_R(y)|}{d(x,y)^{\frac{s}{p}}} \geq \sigma \lambda \bigg\}.
\end{equation}
Since $u = u_R + v_R$, we have $E_\lambda \subset A_1 \cup A_2$. We need to estimate the measures of $A_1$ and $A_2$; we start with the estimate from above.  Since $u_R$ is boundedly supported, we may apply the result obtained in Step 2 with $(1-\sigma)^{-1}u_R$ in place of $u$ and get
\begin{equation}
\lim_{\lambda\to 0}\lambda^p (\nu \otimes \nu)(A_1)= \frac{2 \cdot {\rm AVR}}{(1-\sigma)^p} \| u_R \|_{L^p(X,\nu)}^p .
\end{equation}
To estimate the measure of $A_2$, we will use the second inequality in \eqref{eq:bounds}, which we proved in Step~1. Because the value for any $\lambda$ can be estimated from above by the supremum, we plug in $\sigma\lambda$ and $v_R$ to get
\begin{equation}
\lambda^p \sigma^p (\nu \otimes \nu)(A_2) \leq c_2 \| v_R \|_{L^p(X,\nu)}^p,
\end{equation}
so
\begin{equation}
\limsup_{\lambda\to 0}\lambda^p (\nu \otimes \nu)(A_2) \leq \frac{c_2}{\sigma^p} \| v_R \|_{L^p(X,\nu)}^p.
\end{equation}
We combine the estimates for $A_1$ and $A_2$ to get
\begin{equation}
\limsup_{\lambda \rightarrow 0^+} \lambda^p (\nu \otimes \nu)(E_\lambda) \leq \frac{2 \cdot {\rm AVR}}{(1-\sigma)^p} \| u_R \|_{L^p(X,\nu)}^p + \frac{c_2}{\sigma^p} \| v_R \|_{L^p(X,\nu)}^p.
\end{equation}
In this inequality, let $R \rightarrow \infty$. Since $u_R \rightarrow u$ and $v_R \rightarrow 0$ in $L^p(X,\nu)$, we have 
\begin{equation}
\limsup_{\lambda \rightarrow 0^+} \lambda^p (\nu \otimes \nu)(E_\lambda) \leq \frac{2 \cdot {\rm AVR}}{(1-\sigma)^p} \| u \|_{L^p(X,\nu)}^p.
\end{equation}
Finally, we let $\sigma \rightarrow 0^+$ and get
\begin{equation}\label{cip}
\limsup_{\lambda \rightarrow 0^+} \lambda^p (\nu \otimes \nu)(E_\lambda) \leq 2 \cdot {\rm AVR} \| u \|_{L^p(X,\nu)}^p.
\end{equation}
For the estimate from below, we introduce a third set
\begin{equation}
A_3 = \bigg\{ (x,y) \in X \times X: \frac{|u_R(x) - u_R(y)|}{d(x,y)^{\frac{s}{p}}} \geq (1+\sigma) \lambda \bigg\}.
\end{equation}
Again, we may apply the estimate of Step 2, this time with $u_R$ in place of $u$ and $\lambda(1 + \sigma)$ in place of $\lambda$ and get
\begin{equation}
\lim_{\lambda\to0}\lambda^p (\nu \otimes \nu)(A_3) = \frac{2 \cdot {\rm AVR}}{ (1+\sigma)^p} \| u_R \|_{L^p(X,\nu)}^p.
\end{equation}
Now, notice that $A_3 \backslash A_2 \subset E_\lambda$. Therefore,
\begin{align}
\liminf_{\lambda \rightarrow 0^+} \lambda^p (\nu \otimes \nu)(E_\lambda) &\geq \liminf_{\lambda\to0}\lambda^p (\nu \otimes \nu)(A_3) - \limsup_{\lambda\to0}\lambda^p (\nu \otimes \nu)(A_2) \\
&\qquad\qquad\qquad\qquad \geq \frac{2 \cdot {\rm AVR}}{(1+\sigma)^p} \| u_R \|_{L^p(X,\nu)}^p - \frac{c_2}{\sigma^p} \| v_R \|_{L^p(X,\nu)}^p.
\end{align}
Similarly to the previous estimate, let $R \rightarrow \infty$. Since $u_R \rightarrow u$ and $v_R \rightarrow 0$ in $L^p(X,\nu)$, we have 
\begin{equation}
\liminf_{\lambda \rightarrow 0^+} \lambda^p (\nu \otimes \nu)(E_\lambda) \geq \frac{2 \cdot {\rm AVR}}{(1+\sigma)^p} \| u \|_{L^p(X,\nu)}^p.
\end{equation}
Finally, we let $\sigma \rightarrow 0^+$ and get
\begin{equation}\label{ciop}
\liminf_{\lambda \rightarrow 0^+} \lambda^p (\nu \otimes \nu)(E_\lambda) \geq 2 \cdot {\rm AVR} \| u \|_{L^p(X,\nu)}^p.
\end{equation}
Combining together \eqref{cip} and \eqref{ciop}, we prove \eqref{eq:limit}.
\end{proof}

Below, we state a second version of our main result, with somewhat different assumptions on the measure $\nu$: with respect to Theorem \ref{thm:maintheorem1}, in place of the condition involving the asymptotic volume ratio, we assume that the measure is Ahlfors regular (and not only upper Ahlfors regular).

\begin{theorem}\label{thm:maintheorem2}
Suppose that $\nu(X) = + \infty$. Assume that $\nu$ is Ahlfors regular, namely there exist constants $C_a, C_A > 0$ such that
\begin{equation}\label{eq:ahlforsregular}
C_a r^s \leq \nu(B(x,r)) \leq C_A r^s\ \ \ \forall \ x\in X, \, r>0.
\end{equation}
Then, there exist constants $c_1, c_2 > 0$ such that for all $p \in [1,\infty)$ and $u \in L^p(X,\nu)$ we have
\begin{equation}\label{eq:boundsv2}
c_1 \| u \|_{L^p(X,\nu)}^p \leq \bigg[ \frac{u(x)-u(y)}{d(x,y)^{\frac{s}{p}}} \bigg]^p_{L^p_w(X \times X, \nu \otimes \nu)} \leq c_2 \| u \|_{L^p(X,\nu)}^p.
\end{equation}
We may take $c_1 = 2 C_a$ and $c_2 = 2^{p+1} C_A$.
Furthermore, we have
\begin{equation}\label{eq:limitv2}
2 \cdot C_a \| u \|_{L^p(X,\nu)}^p\le\liminf_{\lambda \rightarrow 0^+} \lambda^{p} (\nu \otimes \nu) (E_\lambda)\le \limsup_{\lambda \rightarrow 0^+} \lambda^{p} (\nu \otimes \nu) (E_\lambda) \le 2 \cdot C_A \| u \|_{L^p(X,\nu)}^p.
\end{equation}
\end{theorem}
\begin{remark}
Notice that if $C_a\equiv C_A$ then we conclude that 
\[
\lim_{\lambda \rightarrow 0^+} \lambda^{p} (\nu \otimes \nu) (E_\lambda) = 2 \cdot C_A \| u \|_{L^p(X,\nu)}^p.
\]
If the two constants $C_a, C_A$ are different and no assumption is made on the AVR, we cannot infer in general that the limit above exists (see Example \ref{ex:nolimitaslambdagoestozero} below).
\end{remark}

\begin{proof}[Proof of Theorem \ref{thm:maintheorem2}]
Observe that the proof of the upper bound in Theorem \ref{thm:maintheorem1} used only upper Ahlfors regularity of the measure $\nu$; therefore, the estimate from above in \eqref{eq:boundsv2} holds with the constant $c_2 = 2^{p+1} C_A$. We now show how to adapt the rest of the proof of Theorem \ref{thm:maintheorem1} to obtain the lower bound under the current assumptions.

{\flushleft \it Step 1.} In order to prove the lower bound in \eqref{eq:boundsv2}, we will again consider the limit of the expression $\lambda^p (\nu \otimes \nu)(E_\lambda)$ as $\lambda \rightarrow 0^+$, as it is smaller or equal to the supremum. Again, we first consider the case when $u$ has bounded support (notice that, under the current assumptions, this is equivalent to ask for compact support). Fix $x_0 \in X$ and $R > 0$ such that $\mbox{supp } u \subset B(x_0,R)$. Working as in Step 2 of the proof of Theorem \ref{thm:maintheorem1}, we see that if~we~denote
\begin{equation}
H_\lambda = E_\lambda \cap \{ (x,y) \in X \times X: \, d(x_0,y) > d(x_0,x) \},
\end{equation}
we have
\begin{equation}\label{eq:twicev2}
(\nu \otimes \nu)(E_\lambda) = 2(\nu \otimes \nu)(H_\lambda),
\end{equation}
so it suffices to compute the limit with $H_\lambda$ in place of $E_\lambda$ (note that an Ahlfors regular measure is doubling, so condition \eqref{eq:continuity} is satisfied). We also use the same definitions of the sets $H_{\lambda,x}$ and $H_{\lambda,x,R}$ as in the
 previous proof, i.e., for $(x,y) \in H_\lambda$ we have that $x \in B(x_0,R)$ and then set
\begin{equation}
H_{\lambda,x} := \bigg\{ y \in X: \, d(x_0,y) > d(x_0,x), \, \frac{|u(x)-u(y)|}{d(x,y)^{\frac{s}{p}}} \geq \lambda \bigg\}
\end{equation}
and
\begin{equation}
H_{\lambda,x,R} = \bigg\{ y \in X: \, d(x_0,y) \geq R, \, d(x,y) \leq \bigg(\frac{|u(x)|}{\lambda} \bigg)^{\frac{p}{s}} \bigg\},
\end{equation}
so that
\begin{equation}\label{eq:inclusionsv2}
H_{\lambda,x,R} \subset H_{\lambda,x} \subset H_{\lambda,x,R} \cup B(x_0,R).
\end{equation}
We now estimate from above and below the measures of $H_{\lambda,x}$ for all $x \in B(x_0,R)$ using \eqref{eq:inclusionsv2}. From the first inclusion, using the Ahlfors regularity condition \eqref{eq:ahlforsregular}, we have
\begin{equation}
\nu(H_{\lambda,x}) \geq \nu(H_{\lambda,x,R}) \geq \nu \bigg(B \bigg(x, \bigg(\frac{|u(x)|}{\lambda} \bigg)^{\frac{p}{s}} \bigg) \bigg) - \nu(B(x_0,R))\geq C_a \frac{|u(x)|^p}{\lambda^p} - C_A R^s.
\end{equation}
Integrating it over $B(x_0,R)$ we get
\begin{equation}
(\nu \otimes \nu)(H_\lambda) = \int_{X} \nu(H_{\lambda,x}) \, d\nu(x) = \int_{B(x_0,R)} \nu(H_{\lambda,x}) \, d\nu(x) \geq \int_{B(x_0,R)} \bigg( C_a \frac{|u(x)|^p}{\lambda^p} - C_A R^s \bigg) \, d\nu(x).
\end{equation}
Multiplying the above inequality by $\lambda^p$ and noticing that the support of $u$ lies in $B(x_0,R)$, we get 
\begin{equation}
\lambda^p (\nu \otimes \nu)(H_\lambda)  \geq C_a \int_X |u(x)|^p \, d\nu(x)- \lambda^p C_A^2 R^{2s}.
\end{equation}
Letting $\lambda \rightarrow 0$, we obtain
\begin{equation}
\liminf_{\lambda \rightarrow 0^+}\lambda^p (\nu \otimes \nu)(H_\lambda) \geq C_a \int_X |u(x)|^p d\nu(x).
\end{equation}
In light of \eqref{eq:twicev2}, this proves the lower bound in \eqref{eq:boundsv2} for boundedly supported $u$ with $c_1 = 2C_a$, i.e.,
\begin{equation}\label{eq:lowerboundforboundedsupport}
\bigg[ \frac{u(x)-u(y)}{d(x,y)^{\frac{s}{p}}} \bigg]_{L^p_w(X \times X, \nu \otimes \nu)} \geq \liminf_{\lambda \rightarrow 0^+} \lambda^p (\nu \otimes \nu)(E_\lambda) \geq 2 C_a \int_X |u(x)|^p \, d\nu(x).
\end{equation}
Furthermore, from the second inclusion in \eqref{eq:inclusionsv2}, we get that
\begin{equation}\label{eq:estimatefromabovev2}
\nu(H_{\lambda,x}) \leq \nu(H_{\lambda,x,R}) + \nu(B(x_0,R)) \leq \nu \bigg(B \bigg(x, \bigg(\frac{|u(x)|}{\lambda} \bigg)^{\frac{p}{s}} \bigg) \bigg) + C_A R^s \leq C_A \frac{|u(x)|^p}{\lambda^p} + C_A R^s,
\end{equation}
which implies
\begin{equation}
\limsup_{\lambda \rightarrow 0^+}\lambda^p (\nu \otimes \nu)(H_\lambda) \leq C_A \int_X |u(x)|^p \, d\nu(x).
\end{equation}
Again using \eqref{eq:twicev2}, we obtain that
\begin{equation}\label{eq:upperboundforboundedsupport}
\limsup_{\lambda \rightarrow 0^+} \lambda^p (\nu \otimes \nu)(E_\lambda) \leq 2 C_A \int_X |u(x)|^p \, d\nu(x).
\end{equation}
Collecting \eqref{eq:lowerboundforboundedsupport} and \eqref{eq:upperboundforboundedsupport}, we obtain that \eqref{eq:limitv2} holds whenever $u$ has bounded support.

{\it \flushleft Step 2.} We now consider general functions $u$ which are not necessarily boundedly supported, relying on the estimate in the previous case and estimating the error from taking a boundedly supported approximation. Denote
\begin{equation}
u_R = u \cdot \1_{B(x_0,R)}
\end{equation}
and
\begin{equation}
v_R = u - u_R.
\end{equation}
Since $u \in L^p(X,\nu)$, both functions also lie in $L^p(X,\nu)$ and $v_R \rightarrow 0$ in $L^p(X,\nu)$ as $R \rightarrow \infty$. Fix $\sigma \in (0,1)$. Then, a simple modification of the argument from Step 3 of the proof of Theorem \ref{thm:maintheorem1} yields the result. We use the same notation, i.e.,
\begin{equation}
A_1 = \bigg\{ (x,y) \in X \times X: \frac{|u_R(x) - u_R(y)|}{d(x,y)^{\frac{s}{p}}} \geq (1-\sigma) \lambda \bigg\},
\end{equation}
\begin{equation}
A_2 = \bigg\{ (x,y) \in X \times X: \frac{|v_R(x) - v_R(y)|}{d(x,y)^{\frac{s}{p}}} \geq \sigma \lambda \bigg\},
\end{equation}
and
\begin{equation}
A_3 = \bigg\{ (x,y) \in X \times X: \frac{|u_R(x) - u_R(y)|}{d(x,y)^{\frac{s}{p}}} \geq (1+\sigma) \lambda \bigg\}.
\end{equation}
First we prove the lower bound; to this end, observe that by the same argument as in the proof of Theorem \ref{thm:maintheorem1} we have
\begin{equation}\label{eq:estimatefora2v2}
\limsup_{\lambda\to 0}\lambda^p (\nu \otimes \nu)(A_2) \leq \frac{c_2}{\sigma^p} \| v_R \|_{L^p(X,\nu)}^p.
\end{equation}
Since $u_R$ is boundedly supported, applying \eqref{eq:lowerboundforboundedsupport} with $u_R$ in place of $u$ and $\lambda(1 + \sigma)$ in place of $\lambda$, we get 
\begin{equation}
\liminf_{\lambda\to0}\lambda^p (\nu \otimes \nu)(A_3) \geq \frac{2 C_a}{ (1+\sigma)^p} \| u_R \|_{L^p(X,\nu)}^p.
\end{equation}
Since $A_3 \backslash A_2 \subset E_\lambda$, we have
\begin{align}
\liminf_{\lambda \rightarrow 0^+} \lambda^p (\nu \otimes \nu)(E_\lambda) &\geq \liminf_{\lambda \to 0}\lambda^p (\nu \otimes \nu)(A_3) - \limsup_{\lambda\to0}\lambda^p (\nu \otimes \nu)(A_2) \\
&\qquad\qquad\qquad\qquad \geq \frac{2 C_a}{(1+\sigma)^p} \| u_R \|_{L^p(X,\nu)}^p - \frac{c_2}{\sigma^p} \| v_R \|_{L^p(X,\nu)}^p.
\end{align}
We now let $R \rightarrow \infty$. Since $u_R \rightarrow u$ and $v_R \rightarrow 0$ in $L^p(X,\nu)$, we have 
\begin{equation}
\liminf_{\lambda \rightarrow 0^+} \lambda^p (\nu \otimes \nu)(E_\lambda) \geq \frac{2 C_a}{(1+\sigma)^p} \| u \|_{L^p(X,\nu)}^p.
\end{equation}
Letting $\sigma \rightarrow 0^+$, we get
\begin{equation}\label{ciopv2}
\liminf_{\lambda \rightarrow 0^+} \lambda^p (\nu \otimes \nu)(E_\lambda) \geq 2 C_a \| u \|_{L^p(X,\nu)}^p,
\end{equation}
which proves the lower bound in \eqref{eq:boundsv2} with $c_1 = 2C_a$.

For the upper bound, applying \eqref{eq:upperboundforboundedsupport} for the function $(1-\sigma)^{-1}u_R$ in place of $u$ yields 
\begin{equation}\label{eq:estimatefora1v2}
\limsup_{\lambda\to 0} \lambda^p (\nu \otimes \nu)(A_1) \leq \frac{2 C_A}{(1-\sigma)^p} \| u_R \|_{L^p(X,\nu)}^p .
\end{equation}
Note that $E_\lambda \subset A_1 \cup A_2$. Thus, combining the estimates \eqref{eq:estimatefora2v2} and \eqref{eq:estimatefora1v2}, we obtain
\begin{equation}
\limsup_{\lambda \rightarrow 0^+} \lambda^p (\nu \otimes \nu)(E_\lambda) \leq \frac{2 C_A}{(1-\sigma)^p} \| u_R \|_{L^p(X,\nu)}^p + \frac{c_2}{\sigma^p} \| v_R \|_{L^p(X,\nu)}^p.
\end{equation}
Letting $R \rightarrow \infty$ results in
\begin{equation}
\limsup_{\lambda \rightarrow 0^+} \lambda^p (\nu \otimes \nu)(E_\lambda) \leq \frac{2 C_A}{(1-\sigma)^p} \| u \|_{L^p(X,\nu)}^p,
\end{equation}
and finally sending $\sigma \rightarrow 0^+$ gives
\begin{equation}
\limsup_{\lambda \rightarrow 0^+} \lambda^p (\nu \otimes \nu)(E_\lambda) \leq 2 C_A \| u \|_{L^p(X,\nu)}^p.
\end{equation}
This together with \eqref{ciopv2} shows \eqref{eq:limitv2} in the general case.
\end{proof}

Actually, the proofs presented in this Section show us a bit more. If instead of a power of the distance function, equal to the dimension $s$ in the Ahlfors regularity condition, we take a more general function of the distance, we easily obtain similar results. The two results presented below are corollaries of the proofs of Theorem \ref{thm:maintheorem1} and \ref{thm:maintheorem2}. To this end, we need to generalise the Ahlfors regularity condition and the asymptotic value ratio.
 
\begin{theorem}\label{thm:cor1}
Suppose that $\nu(X) = + \infty$. Let $f: [0,\infty) \rightarrow [0,\infty)$ be a convex increasing function with $f(0) = 0$. Assume that there exists $C_{f,A} > 0$ such that
\begin{equation}\label{eq:upperfahlfors}
\nu(B(x,r)) \leq C_{f,A} f(r) \ \ \ \forall \ x\in X, \, r>0
\end{equation}
(we say that $\nu$ is {\it upper $f$-Ahlfors regular}). Furthermore, we require that condition \eqref{eq:continuity} holds, and that the limit
\begin{equation}\label{eq:favr}
{\rm AVR}_f = \lim_{r \rightarrow \infty} \frac{\nu(B(x_0,r))}{f(r)}
\end{equation}
is well-defined for some (equivalently: all) $x_0 \in X$ and ${\rm AVR}_f \in (0,+\infty)$. Then, there exist constants $c_1, c_2, c_3 > 0$ such that for all $p \in [1,\infty)$ and $u \in L^p(X,\nu)$ we have
\begin{equation}\label{eq:bounds3}
c_1 \| u \|_{L^p(X,\nu)}^p \leq \bigg[ \frac{u(x)-u(y)}{f(d(x,y))^{\frac{1}{p}}} \bigg]^p_{L^p_w(X \times X, \nu \otimes \nu)} \leq c_2 \| u \|_{L^p(X,\nu)}^p.
\end{equation}
Moreover, the lower bound can be improved in the following way. If we denote
\begin{equation}
E_{\lambda}^f = \bigg\{(x,y)\in X \times X: \, \frac{|u(x)-u(y)|}{f(d(x,y))^{\frac{1}{p}}} \geq \lambda \bigg\}
\end{equation}
we have
\begin{equation}\label{eq:limit3}
\lim_{\lambda \rightarrow 0^+} \lambda^{p} (\nu \otimes \nu) (E_\lambda^f) = c_3 \| u \|_{L^p(X,\nu)}^p.
\end{equation}
We may take $c_1 = c_3 = 2 \cdot {\rm AVR}_f$ and $c_2 = 2^{p+1} C_{f,A}$. 
\end{theorem}

Arguing as in \cite{HP}, it is easy to see that ${\rm AVR}_f$ does not depend on $x_0$. For the choice $f(r) = r^s$, we recover the statement of Theorem \ref{thm:maintheorem1}; for the choice $f(r) = e^r$, we retrieve a variant of the volume entropy condition appearing in \cite{HP}; for general $f$, the main idea behind this result is to allow for a superlinear growth of the volume of the balls which is faster than any power-type growth.

\begin{proof}
The outline of the proof is identical to the one of Theorem \ref{thm:maintheorem1}. For the upper bound in Step 1, it is enough to see that when $(x,y) \in E_\lambda^f$, then either $|u(x)| \geq \frac12 \lambda f(d(x,y))^{\frac{1}{p}}$ or $|u(y)| \geq \frac12 \lambda f(d(x,y))^{\frac{1}{p}}$. Then, by the upper $f$-Ahlfors regularity of $f$ (condition \eqref{eq:upperfahlfors}),
\begin{align}
\lambda^p \int_X \int_X &\1_{\{ (x,y): \, |u(x)| \geq \frac12 \lambda f(d(x,y))^{\frac{1}{p}} \} } \, d\nu(y) \, d\nu(x) \\
&\qquad\qquad\qquad = \lambda^p \int_X \int_X \1_{\{ (x,y): \, d(x,y) \leq f^{-1}(\frac{2^p}{\lambda^p} |u(x)|^p) \} } \, d\nu(y) \, d\nu(x) \\
&\qquad\qquad\qquad = \lambda^p \int_X \nu \bigg(B \bigg (x, f^{-1}\bigg(\frac{2^p}{\lambda^p} |u(x)|^p \bigg) \bigg) \bigg) \, d\nu(x) \leq 2^p C_{f,A} \int_X |u(x)|^p \, d\nu(x),
\end{align}
and we conclude as in Theorem \ref{thm:maintheorem1} that the upper bound in \eqref{eq:bounds3} holds 
with $c_2 = 2^{p+1} C_{f,A}$.

For the lower bound in \eqref{eq:bounds3} for boundedly supported functions in Step 2, we again compute the limit in \eqref{eq:limit3}. Fix $x_0 \in X$ and $R > 0$ such that $\mbox{supp } u \subset B(x_0,R)$. It is enough to compute the measure of the set
\begin{equation}
H_\lambda^f = E_\lambda^f \cap \{ (x,y) \in X \times X: \, d(x_0,y) > d(x_0,x) \}.
\end{equation}
For $x \in B(x_0,R)$, setting
\begin{equation}
H_{\lambda,x}^f := \bigg\{ y \in X: \, d(x_0,y) > d(x_0,x), \, \frac{|u(x)-u(y)|}{f(d(x,y))^{\frac{1}{p}}} \geq \lambda \bigg\}
\end{equation}
and
\begin{equation}
H_{\lambda,x,R}^f = \bigg\{ y \in X: \, d(x_0,y) \geq R, \, d(x,y) \leq f^{-1}\bigg(\frac{|u(x)|^p}{\lambda^p} \bigg) \bigg\},
\end{equation}
so that inclusions \eqref{eq:inclusions} hold with the modified sets $H_{\lambda,x}^f$ and $H_{\lambda,x,R}^f$, we estimate
\begin{equation}
\nu(H_{\lambda,x,R}^f) \geq \nu \bigg(B \bigg(x, f^{-1}\bigg(\frac{|u(x)|^p}{\lambda^p} \bigg) \bigg) \bigg) - \nu(B(x_0,R))\geq  \nu \bigg(B \bigg(x, f^{-1}\bigg(\frac{|u(x)|^p}{\lambda^p} \bigg) \bigg) \bigg) - C_{f,A} f(R),
\end{equation}
and we conclude as in the proof of Theorem \ref{thm:maintheorem1} that
\begin{equation}
\lambda^p (\nu \otimes \nu)(H_\lambda^f)  \geq \lambda^p \int_X \nu \bigg( B \bigg(x, f^{-1}\bigg(\frac{|u(x)|^p}{\lambda^p} \bigg) \bigg) \bigg) \, d\nu(x)- \lambda^p C_{f,A}^2 f(R)^{2}.
\end{equation}
To take the limit as $\lambda\to 0$, define the function 
$$f_{\lambda}(x) =\lambda^p \cdot \nu \bigg(B \bigg(x, f^{-1}\bigg(\frac{|u(x)|^p}{\lambda^p} \bigg) \bigg) \bigg)$$
and observe that
\begin{equation}\label{eq:fflambda}
0\leq f_{\lambda}(x)\leq C_{f,A} |u(x)|^p \qquad \mbox{and} \qquad f_{\lambda}\to {\rm AVR}_f \cdot |u(x)|^p \quad \nu-\mbox{a.e} \ \mbox{in } X;
\end{equation}
indeed, by condition \eqref{eq:upperfahlfors}
\begin{equation}
f_\lambda(x) = \lambda^p \cdot \nu \bigg(B \bigg(x, f^{-1}\bigg(\frac{|u(x)|^p}{\lambda^p} \bigg) \bigg) \bigg) \leq \lambda^p C_{f,A} f \bigg( f^{-1}\bigg(\frac{|u(x)|^p}{\lambda^p} \bigg) \bigg) = C_{f,A} |u(x)|^p
\end{equation}
and making the change of variables $r= f^{-1}(\frac{|u(x)|^p}{\lambda^p})$, by assumption \eqref{eq:favr} we obtain
\begin{equation}
\lim_{\lambda \rightarrow 0^+} f_\lambda(x) = \lim_{\lambda \rightarrow 0^+} \lambda^p \cdot \nu \bigg(B \bigg(x, f^{-1}\bigg(\frac{|u(x)|^p}{\lambda^p} \bigg) \bigg) \bigg) = \lim_{r \rightarrow \infty} \frac{|u(x)|^p}{f(r)} \nu(B(x,r)) = {\rm AVR}_f \cdot |u(x)|^p.
\end{equation}
Using the properties \eqref{eq:fflambda}, we conclude the proof of claim \eqref{eq:limit3} for boundedly supported functions as in Step 2 of the proof of Theorem \ref{thm:maintheorem1}.

Finally, the passage from boundedly supported functions to the general case in Step 3 is the same up to replacing $E_\lambda$ with $E_\lambda^f$ and ${\rm AVR}$ with ${\rm AVR}_f$ (in the whole argument), as well as $d(x,y)^{\frac{s}{p}}$ by $f(d(x,y))^{\frac{1}{p}}$ in the definitions of sets $A_1$, $A_2$ and $A_3$.
\end{proof}

Using a similar argument, we recover also the following variant of Theorem \ref{thm:maintheorem2}.

\begin{theorem}\label{thm:cor2}
Suppose that $\nu(X) = + \infty$. Let $f: [0,\infty) \rightarrow [0,\infty)$ be a convex increasing function with $f(0) = 0$. Assume that there exist $C_{f,a}, C_{f,A} > 0$ such that
\begin{equation}\label{eq:fahlfors}
C_{f,a} f(r) \leq \nu(B(x,r)) \leq C_{f,A} f(r) \ \ \ \forall \ x\in X, \, r>0
\end{equation}
(we say that $\nu$ is {\it $f$-Ahlfors regular}). Then, there exist constants $c_1, c_2 > 0$ such that for all $p \in [1,\infty)$ and $u \in L^p(X,\nu)$ we have
\begin{equation}\label{eq:bounds4}
c_1 \| u \|_{L^p(X,\nu)}^p \leq \bigg[ \frac{u(x)-u(y)}{f(d(x,y))^{\frac{1}{p}}} \bigg]^p_{L^p_w(X \times X, \nu \otimes \nu)} \leq c_2 \| u \|_{L^p(X,\nu)}^p.
\end{equation}
We may take $c_1 = 2 C_{f,a}$ and $c_2 = 2^{p+1} C_{f,A}$. Furthermore, if we denote
\begin{equation}
E_{\lambda}^f = \bigg\{(x,y)\in X \times X: \, \frac{|u(x)-u(y)|}{f(d(x,y))^{\frac{1}{p}}} \geq \lambda \bigg\}
\end{equation}
we have
\begin{equation}\label{eq:limit4}
2 \cdot C_{f,a} \| u \|_{L^p(X,\nu)}^p\le\liminf_{\lambda \rightarrow 0^+} \lambda^{p} (\nu \otimes \nu) (E_\lambda^f) \le \limsup_{\lambda \rightarrow 0^+} \lambda^{p} (\nu \otimes \nu) (E_\lambda^f) \le 2 \cdot C_{f,A} \| u \|_{L^p(X,\nu)}^p.
\end{equation}
\end{theorem}

We now briefly discuss the optimality of the assumptions. The following example shows that, in general, one cannot relax the assumption that the measure of $X$ is infinite, as otherwise the lower bound fails (the upper bound still holds if the measure is upper Ahlfors regular).

\begin{example}\label{ex:finitemeasure}
Suppose that the metric measure space $(X,d,\nu)$ is such that $\nu(X) < \infty$. Then,
\begin{equation}
\liminf_{\lambda \rightarrow 0^+} \lambda^p (\nu \otimes \nu)(E_\lambda) \leq \lim_{\lambda \rightarrow 0^+} \lambda^p (\nu \otimes \nu)(X \times X) = 0,
\end{equation}
so the lower bound in \eqref{eq:bounds} fails.
\end{example}

The following example shows that, in general, one cannot relax the assumption of upper Ahlfors regularity and consider general metric spaces equipped with a doubling measure.

\begin{example}\label{ex:counterexample}
Consider the metric measure space 
$$(X,d,\nu) = (\mathbb{R},d_{\rm Eucl}, (1+ |x|) \mathcal{L}^1).$$
It is easy to see that this metric measure space is doubling, but it is not Ahlfors regular: an explicit computation yields
\begin{equation}
\nu(B(x,r)) = \int_{x-r}^{x+r} (1+|t|) \, dt = \threepartdef{2r + 2|x|r}{\mbox{if } \, 0 \leq r < |x|;}{x^2 + r^2 + 2r}{\mbox{if } \, 0 < |x| \leq r;}{r(r+2)}{\mbox{if } \, x = 0.}
\end{equation}
Hence, the ratio $\frac{\nu(B(x,2r))}{\nu(B(x,r))}$ is uniformly bounded, with the doubling constant $C_d = 4$. In particular, its homogeneous dimension equals $s = \log_2 4 = 2$. The measure $\nu$ also admits an asymptotic volume ratio, with exponent $s = 2$, and we have ${\rm AVR} = 1$. On the other hand, the upper bound in the definition of Ahlfors regularity fails for any $s > 0$, since the measure of the ball with a fixed radius goes to infinity as $x \rightarrow +\infty$.

Take any $p \in [1,\infty)$ and $s > 0$. Then, fixing $\lambda = 1$ and taking the sequence $u_n = \1_{[n,n+1]}$, we can estimate from below the Marcinkiewicz seminorm of $u_n$ in the following way:
\begin{align}
\bigg[ \frac{u_n(x)-u_n(y)}{|x-y|^{\frac{s}{p}}} \bigg]_{L^p_w(X \times X, \nu \otimes \nu)} \geq (\nu \otimes \nu)(E_1) =  (\nu \otimes \nu) \bigg(\bigg\{(x,y)\in \mathbb{R} \times \mathbb{R}: \, \frac{|u_n(x)-u_n(y)|}{|x-y|^{\frac{s}{p}}} \geq 1 \bigg\}\bigg),
\end{align}
and since $u$ is a characteristic function, we can further rewrite the right-hand side as
\begin{align}
(\nu \otimes \nu) &\bigg(\bigg\{(x,y)\in \mathbb{R} \times \mathbb{R}: \, \frac{|u_n(x)-u_n(y)|}{|x-y|^{\frac{s}{p}}} \geq 1 \bigg\}\bigg) \\
&\qquad\qquad= (\nu \otimes \nu) \bigg(\bigg\{ x \in [n,n+1], y \notin [n,n+1]: \, \frac{1}{|x-y|^{\frac{s}{p}}} \geq 1 \bigg\}\bigg) \\ &\qquad\qquad\qquad\qquad\qquad\qquad + (\nu \otimes \nu) \bigg(\bigg\{ x \notin [n,n+1], y \in [n,n+1]: \, \frac{1}{|x-y|^{\frac{s}{p}}} \geq 1 \bigg\}\bigg) \\
&\qquad\qquad= (\nu \otimes \nu) \bigg(\bigg\{ x \in [n,n+1], y \notin [n,n+1]: \, |x-y| \leq 1 \bigg\}\bigg) \\ &\qquad\qquad\qquad\qquad\qquad\qquad + (\nu \otimes \nu) \bigg(\bigg\{ x \notin [n,n+1], y \in [n,n+1]: \, |x-y| \leq 1 \bigg\}\bigg).
\end{align}
Therefore,
\begin{align}
(\nu \otimes \nu) &\bigg(\bigg\{(x,y)\in \mathbb{R} \times \mathbb{R}: \, \frac{|u_n(x)-u_n(y)|}{|x-y|^{\frac{s}{p}}} \geq 1 \bigg\}\bigg) \\
&\qquad\qquad\qquad\qquad\qquad \geq (\nu \otimes \nu) \bigg( \bigg[n + \frac12,n+1 \bigg] \times \bigg(n+1,n+\frac32 \bigg] \bigg) \geq \frac{n^2}{4},
\end{align}
and consequently
\begin{align}
\bigg[ \frac{u_n(x)-u_n(y)}{|x-y|^{\frac{s}{p}}} \bigg]_{L^p_w(X \times X, \nu \otimes \nu)} \geq \frac{n^2}{4}.
\end{align}
Yet, we have
\begin{equation}
\| u_n \|_{L^p(X,\nu)}^p = \int_{n}^{n+1} (1 + t) \, dt = n+\frac32 ,
\end{equation}
so there is no uniform upper bound in \eqref{eq:bounds} for this metric measure space and all $u \in L^p(X,\nu)$.
\end{example}

The final example concerns the case when the measure is Ahlfors regular, but does not admit an asymptotic volume ratio. Then, we will see that even though the lower and upper bounds are valid as given in Theorem \ref{thm:maintheorem2}, the limit of $\lambda^p (\nu \otimes \nu)(E_\lambda)$ as $\lambda \rightarrow 0$ is not well-defined.

\begin{example}\label{ex:nolimitaslambdagoestozero}
Consider the metric measure space 
$$(X,d,\nu) = (\mathbb{R},d_{\rm Eucl}, w \mathcal{L}^N),$$
with the weight $w$ defined in the following way. Let $r_0 = 0$ and $0 < m < M$. For an increasing sequence $r_n \rightarrow +\infty$ which will be specified later, define
\begin{equation}
w(x) = \twopartdef{\displaystyle m}{\mbox{on } B(0,r_n) \setminus B(0,r_{n-1}), \, n \mbox{ odd};}{M}{\mbox{on } B(0,r_n) \setminus B(0,r_{n-1}), \, n \mbox{ even}.}
\end{equation}
The measure $w \mathcal{L}^N$ is clearly Ahlfors regular (with dimension $N$), so the assumptions of Theorem \ref{thm:maintheorem2} are satisfied. Let us now pick the sequence $r_n$ in such a way that the asymptotic value ratio does not exist. Fix any $r_1 > 0$, and then for even $n$ choose $r_n$ large enough so that
\begin{equation}\label{eq:evenn1}
\frac{\displaystyle \int_{B(0,r_n)} w \, d\mathcal{L}^N}{r_n^N} \geq \bigg(M - \frac{1}{n} \bigg) \omega_N,
\end{equation}
where $\omega_N$ denotes the Lebesgue measure of the unit ball (it is enough to take $r_n \geq (M n)^{\frac{1}{N}} r_{n-1}$). Similarly, for odd $n$ we choose $r_n$ large enough so that
\begin{equation}\label{eq:oddn1}
\frac{\displaystyle \int_{B(0,r_n)} w \, d\mathcal{L}^N}{r_n^N} \leq \bigg(m + \frac{1}{n}\bigg) \omega_N
\end{equation}
(here, again one may take $r_n \geq (M n)^{1/N} r_{n-1}$). For such a sequence $r_n$, we have
\begin{equation}
\limsup_{r \rightarrow \infty} \frac{\nu(B(x,r))}{r^N} = M \omega_N > m \omega_N = \liminf_{r \rightarrow \infty} \frac{\nu(B(x,r))}{r^N},
\end{equation}
so the metric measure space does not have an asymptotic volume ratio.

Let us see that in this setting the limit $\lambda^p (\nu \otimes \nu)(E_\lambda)$ as $\lambda \rightarrow 0$ needs not be defined for any $p \geq 1$. For simplicity, take $r_1 = 1$ and assume that $r_n \geq (M n)^{\frac{1}{N}} r_{n-1} + 1$, so that the inequalities \eqref{eq:evenn1} and \eqref{eq:oddn1} hold not only for $r_n$, but also for $r_n - 1$. Setting $u = \1_{B(0,1)}$, by a direct computation we have
\begin{align}
(\nu \otimes \nu) &\bigg(\bigg\{(x,y)\in \mathbb{R}^N \times \mathbb{R}^N: \, \frac{|u(x)-u(y)|}{|x-y|^{\frac{N}{p}}} \geq \lambda \bigg\}\bigg) \\
&\qquad\qquad= (\nu \otimes \nu) \bigg(\bigg\{ x \in B(0,1), \, y \notin B(0,1): \, \frac{1}{|x-y|^{\frac{N}{p}}} \geq \lambda \bigg\}\bigg) \\ &\qquad\qquad\qquad\qquad\qquad\qquad + (\nu \otimes \nu) \bigg(\bigg\{ x \notin B(0,1), \, y \in B(0,1): \, \frac{1}{|x-y|^{\frac{N}{p}}} \geq \lambda \bigg\}\bigg) \\
&\qquad\qquad= 2(\nu \otimes \nu) \bigg(\bigg\{ x \in B(0,1), \, y \notin B(0,1): \, \frac{1}{|x-y|^{\frac{N}{p}}} \geq \lambda \bigg\}\bigg) \\
&\qquad\qquad= 2(\nu \otimes \nu) \bigg(\bigg\{ x \in B(0,1), \, y \notin B(0,1): \, |x-y| \leq \lambda^{-\frac{p}{N}} \bigg\}\bigg) \\ &\qquad\qquad= 2 \int_{B(0,1)} \nu(B(x,\lambda^{-\frac{p}{N}}) \setminus B(0,1)) \, d\nu(x).
\end{align}
We now compute the desired limit by estimating the value of the last integral. Taking a decreasing sequence $\lambda_n \rightarrow 0$ such that $r_n = \lambda_n^{-\frac{p}{N}} + 1$, for even $n \geq 2$ we~obtain
\begin{align}
\lambda_n^p (\nu \otimes \nu)(E_{\lambda_n}) &= 2 \lambda_n^p \int_{B(0,1)} \nu(B(x,\lambda_n^{-\frac{p}{N}}) \setminus B(0,1)) \, d\nu(x) \\
&\geq 2 \lambda_n^p \int_{B(0,1)} \bigg( \nu(B(0,r_n - 1)) - \nu(B(0,1)) \bigg) \, d\nu(x) \\
&= 2 m \omega_N \lambda_n^p (\nu(B(0,r_n - 1)) - \nu(B(0,1))) \\
&\geq 2 m \omega_N \lambda_n^p \bigg( \bigg(M - \frac{1}{n} \bigg) \omega_N (r_n - 1)^N - m \omega_N \bigg) \\
&= 2 m \omega_N^2 \lambda_n^p \bigg( \bigg(M - \frac{1}{n} \bigg) \lambda_n^{-p} - m \bigg).
\end{align}
Consequently, considering such a sequence $\lambda_n$ for even $n$, we get
\begin{equation}
\limsup_{\lambda \rightarrow 0^+} \lambda^p (\nu \otimes \nu)(E_\lambda) \geq 2 M m \omega_N^2.
\end{equation}
Similarly, taking a decreasing sequence $\lambda_n \rightarrow 0$ such that $r_n = \lambda_n^{-\frac{p}{N}}$, for odd $n \geq 3$ we obtain
\begin{align}
\lambda_n^p (\nu \otimes \nu)(E_{\lambda_n}) &= 2 \lambda_n^p \int_{B(0,1)} \nu(B(x,\lambda_n^{-\frac{p}{N}}) \setminus B(0,1)) \, d\nu(x) \leq 2 \lambda_n^p \int_{B(0,1)} \nu(B(0,r_n)) \, d\nu(x) \\
&= 2 m \omega_N \lambda_n^p \nu(B(0,r_n)) \leq 2 m \omega_N \lambda_n^p \bigg(m + \frac{1}{n} \bigg) \omega_N r_n^N = 2 m \omega_N^2 \bigg(m - \frac{1}{n} \bigg).
\end{align}
Therefore,
\begin{equation}
\liminf_{\lambda \rightarrow 0^+} \lambda^p (\nu \otimes \nu)(E_\lambda) \leq 2 m^2 \omega_N^2,
\end{equation}
and consequently the limit $\lambda^p (\nu \otimes \nu)(E_\lambda)$ is not well-defined.
\end{example}

\section{Particular cases}\label{sec:applications} 

This section is dedicated to presenting several applications of our main results (Theorem \ref{thm:maintheorem1} and Theorem \ref{thm:maintheorem2}). The first two are already known; for some simple choices of the metric measure space we recover the results of Gu-Yung \cite{GY} and Gu-Huang \cite{GH}. Then, we present how our main Theorems already yield new results in the setting of weighted Euclidean spaces, Carnot groups and Riemannian manifolds. For more examples of spaces satisfying the asymptotic volume ratio, to which we may apply our results, see \cite{HP}.

\subsection{Euclidean space}

For the choice $(X,d,\nu) = (\mathbb{R}^N, d_{\rm Eucl}, \mathcal{L}^N)$, we recover the original result for the weak Maz'ya-Shaposhnikova formula proved in Gu-Yung \cite{GY}. The constants coming from the application of Theorem \ref{thm:maintheorem1} (or Theorem \ref{thm:maintheorem2}) are $c_1 = c_3 = 2 \omega_N$ and $c_2 = 2^{p+1} \omega_N$, where $\omega_N$ denotes the Lebesgue measure of the unit ball. In other words, for all $u \in L^p(\mathbb{R}^N)$ we have
\begin{equation}
2 \omega_N \| u \|_{L^p(\mathbb{R}^N)}^p \leq \bigg[ \frac{u(x)-u(y)}{|x-y|^{\frac{N}{p}}} \bigg]^p_{L^p_w(\mathbb{R}^{N} \times \mathbb{R}^N)} \leq 2^{p+1} \omega_N \| u \|_{L^p(\mathbb{R}^N)}^p.
\end{equation}
Moreover, the set $E_\lambda$ takes the form
\begin{equation}
E_{\lambda}= \bigg\{(x,y)\in \mathbb{R}^N \times \mathbb{R}^N: \ \frac{|u(x)-u(y)|}{|x-y|^{\frac{N}{p}}} \geq \lambda \bigg\}
\end{equation}
and the lower bound is given by
\begin{equation}
\lim_{\lambda \rightarrow 0^+} \lambda^{p} \mathcal{L}^{2N}(E_\lambda) = 2\omega_N \| u \|_{L^p(\mathbb{R}^N)}^p.
\end{equation}

\subsection{Anisotropic norms on $\mathbb{R}^N$}

For the choice $(X,d,\nu) = (\mathbb{R}^N, \| \cdot \|, \mathcal{L}^N)$, where $\| \cdot \|$ is a norm on $\mathbb{R}^N$ whose unit ball is a convex set $K$, we recover the results obtained in Gu-Huang \cite{GH}. The~constants coming from the application of Theorem \ref{thm:maintheorem1} (or Theorem \ref{thm:maintheorem2}) are $c_1 = c_3 = 2 \mathcal{L}^{N}(K)$ and $c_2 = 2^{p+1} \mathcal{L}^{N}(K)$; in other words, for all $u \in L^p(\mathbb{R}^N)$ we have
\begin{equation}
2 \mathcal{L}^{N}(K) \| u \|_{L^p(\mathbb{R}^N)}^p \leq \bigg[ \frac{u(x)-u(y)}{\|x-y\|^{\frac{N}{p}}} \bigg]^p_{L^p_w(\mathbb{R}^{N} \times \mathbb{R}^N)} \leq 2^{p+1} \mathcal{L}^{N}(K) \| u \|_{L^p(\mathbb{R}^N)}^p.
\end{equation}
Moreover, the set $E_\lambda$ takes the form
\begin{equation}
E_{\lambda}= \bigg\{(x,y)\in \mathbb{R}^N \times \mathbb{R}^N: \ \frac{|u(x)-u(y)|}{\|x-y\|^{\frac{N}{p}}} \geq \lambda \bigg\}
\end{equation}
and the lower bound is given by
\begin{equation}
\lim_{\lambda \rightarrow 0^+} \lambda^{p} \mathcal{L}^{2N}(E_\lambda) = 2 \mathcal{L}^{N}(K) \| u \|_{L^p(\mathbb{R}^N)}^p.
\end{equation}

\subsection{Weighted Euclidean spaces}

Let $(X,d,\nu) = (\mathbb{R}^N, d_{\rm Eucl}, w \mathcal{L}^N)$, where $w \in L^\infty(\Omega)$ is a nonnegative weight. If it is bounded away from zero, i.e., $m \leq w(x) \leq M$ for $\mathcal{L}^N$-a.e. $x \in \mathbb{R}^N$, then the measure $w \mathcal{L}^N$ is Ahlfors regular with $C_a = m \omega_N$ and $C_A = M \omega_N$, where $\omega_N$ is again the Lebesgue measure of the unit ball. Therefore, the constants coming from the application of Theorem \ref{thm:maintheorem2} are $c_1 = 2 m \omega_N$ and $c_2 = 2^{p+1} M \omega_N$, and we obtain
\begin{equation}
2 m \omega_N \| u \|_{L^p(\mathbb{R}^N,w\mathcal{L}^N)}^p \leq \bigg[ \frac{u(x)-u(y)}{|x-y|^{\frac{N}{p}}} \bigg]^p_{L^p_w(\mathbb{R}^{N} \times \mathbb{R}^N, w\mathcal{L}^N \otimes w \mathcal{L}^N)} \leq 2^{p+1} M \omega_N \| u \|_{L^p(\mathbb{R}^N, w\mathcal{L}^N)}^p.
\end{equation}
Note that we could obtain a similar result from the one for Euclidean spaces and using the equivalence of the Lebesgue measure with the weighted Lebesgue measure, but applying directly Theorem \ref{thm:maintheorem2} we obtain sharper constants in these inequalities.

Assume instead that the nonnegative weight $w \in L^\infty(\Omega)$ admits a limit at infinity, i.e., the limit
\begin{equation}
w_\infty = \lim_{R \rightarrow \infty} \dashint_{B(0,R)} w \, d\mathcal{L}^N
\end{equation}
is well-defined (observe that we can take any point $x \in \mathbb{R}^N$ in place of the origin). Then, the measure $w \mathcal{L}^N$ is upper Ahlfors regular with $C_A = M \omega_N$, where $M = \| w \|_\infty$, and it satisfies the asymptotic volume ratio (in the sense of Definition \ref{dfn:avr}) with ${\rm AVR} = w_\infty \omega_N$. Condition \eqref{eq:continuity} is satisfied since $w\mathcal{L}^N \ll \mathcal{L}^N$. Therefore, the constants coming from the application of Theorem \ref{thm:maintheorem1} are $c_1 = c_3 = 2 w_\infty \omega_N$ and $c_2 = 2^{p+1} M \omega_N$, and we obtain
\begin{equation}
2 w_\infty \omega_N \| u \|_{L^p(\mathbb{R}^N,w\mathcal{L}^N)}^p \leq \bigg[ \frac{u(x)-u(y)}{|x-y|^{\frac{N}{p}}} \bigg]^p_{L^p_w(\mathbb{R}^{N} \times \mathbb{R}^N, w\mathcal{L}^N \otimes w \mathcal{L}^N)} \leq 2^{p+1} M \omega_N \| u \|_{L^p(\mathbb{R}^N, w\mathcal{L}^N)}^p.
\end{equation}
Observe that this setting does not require that the weight $w$ is bounded from below by a nonnegative number, and in particular it may vanish on subsets of $\mathbb{R}^N$.

Furthermore, we may consider the case when the weight is not necessarily absolutely continuous with respect to the Lebesgue measure, i.e., let $(X,d,\nu) = (\mathbb{R}^N, d_{\rm Eucl}, \nu)$ with $\nu$ which is a general Radon measure satisfying the assumptions of Theorem \ref{thm:maintheorem1}. As a simple example, consider a affine subspace $V \subset \mathbb{R}^N$ of dimension $k$ and $\nu = \mathcal{H}^k$, where $\mathcal{H}^k$ denotes the Hausdorff measure of dimension $k$; it is upper Ahlfors regular of dimension $k$, it satisfies \eqref{eq:continuity} and admits an asymptotic volume ratio (equal to $\omega_k$). One may also easily construct nontrival subsets $V$ of dimension $k$ which satisfy these properties.

\subsection{Carnot groups}

Let $(X,d,\nu) = (\mathbb{G}, d_{\rm cc}, \mathcal{L}^Q)$, where $\mathbb{G}$ is a Carnot group equipped with the Carnot-Carath\'eodory distance $d_{\rm cc}$ and the invariant measure $\mathcal{L}^Q$, where $Q$ is the homogeneous dimension of $\mathbb{G}$ (see the survey \cite{LeD} and the references therein). Then, the Lebesgue measure $\mathcal{L}^Q$ is Ahlfors regular with $C_a = C_A = \mathcal{L}^Q(B_{\rm cc}(0,1))$, where $B_{\rm cc}(0,1)$ is the unit ball in the Carnot-Carath\'eodory distance. Thus, the constants coming from the application of Theorem \ref{thm:maintheorem1} (or Theorem \ref{thm:maintheorem2}) are $c_1 = c_3 = 2 \mathcal{L}^Q(B_{\rm cc}(0,1))$ and $c_2 = 2^{p+1} \mathcal{L}^Q(B_{\rm cc}(0,1))$; in other words, for all $u \in L^p(\mathbb{G}, \mathcal{L}^Q)$ we have
\begin{equation}
2 \mathcal{L}^Q(B_{\rm cc}(0,1)) \| u \|_{L^p(\mathbb{G}, \mathcal{L}^Q)}^p \leq \bigg[ \frac{u(x)-u(y)}{d_{\rm cc}(x,y)^{\frac{Q}{p}}} \bigg]^p_{L^p_w(\mathbb{G} \times \mathbb{G}, \mathcal{L}^{2Q})} \leq 2^{p+1} \mathcal{L}^Q(B_{\rm cc}(0,1)) \| u \|_{L^p(\mathbb{G}, \mathcal{L}^Q)}^p.
\end{equation}
Moreover, the set $E_\lambda$ takes the form
\begin{equation}
E_{\lambda}= \bigg\{(x,y)\in \mathbb{G} \times \mathbb{G}: \ \frac{|u(x)-u(y)|}{d_{\rm cc}(x,y)^{\frac{Q}{p}}} \geq \lambda \bigg\}
\end{equation}
and the lower bound is given by
\begin{equation}
\lim_{\lambda \rightarrow 0^+} \lambda^{p} \mathcal{L}^{2Q}(E_\lambda) = 2 \mathcal{L}^Q(B_{\rm cc}(0,1)) \| u \|_{L^p(\mathbb{G}, \mathcal{L}^Q)}^p.
\end{equation}

\subsection{Riemannian manifolds}

Let $(M,g)$ be a complete Riemannian manifold of dimension $N$, and denote by $(X,d,\nu) = (M, d_{M}, {\rm Vol})$ the metric measure space where $d_{M}$ and ${\rm Vol}$ are respectively the geodesic distance and the volume density prescribed by the Riemannian metric $g$. Then, the classical Bishop-Gromov theorem implies that whenever $M$ has nonnegative Ricci curvature, for all $x \in M$ the ratio
\begin{equation}
\frac{{\rm Vol}(B(x,r))}{\omega_N r^N}
\end{equation}
is nonincreasing in $r$ ($\omega_N$ again denotes the Lebesgue measure of the unit ball in $\mathbb{R}^N$). Therefore, it has a limit as $r \rightarrow \infty$, and it is easy to see that it does not depend on the choice of $x$. Furthermore, it has limit equal to one as $r \rightarrow 0$. Thus, the measure ${\rm Vol}$ is upper Ahlfors regular of dimension~$N$, it gives no mass to spheres, and if the limit
\begin{equation}
\lim_{r \rightarrow \infty} \frac{{\rm Vol}(B(x,r))}{\omega_N r^N}
\end{equation}
is positive, it satisfies the asymptotic volume ratio. Then, we may apply Theorem \ref{thm:maintheorem1} to obtain a corresponding weak Maz'ya-Shaposhnikova formula.

\subsection{Spaces with synthetic Ricci curvature bounds} Let $(X,d,\nu)$ be a $CD(K,N)$ space, i.e., a metric measure space which satisfies the curvature-dimension condition introduced independently by Sturm \cite{Sturm,Sturm2} and Lott-Villani \cite{LV}. Assume that $K = 0$ and $N \geq 1$. In a generalised sense, it is a space with nonnegative Ricci curvature and dimension bounded from above by $N$. Then, if $\nu$ is not a Dirac measure, by a generalisation of the Bishop-Gromov theorem given in \cite[Theorem 2.3]{Sturm2}, the measure $\nu$ assigns no mass to spheres, and the ratio
\begin{equation}
\frac{\nu(B(x,r))}{r^N}
\end{equation}
is nonincreasing in $r$ (for $N = 1$, one may also take $K \leq 0$). Thus, it has a limit as $r \rightarrow \infty$, which does not depend on the choice of $x$. Then, if the limits of this quotient as $r \rightarrow 0$ and $r \rightarrow \infty$ are finite, the measure $\nu$ is respectively upper Ahlfors regular and admits an asymptotic volume ratio, and we may apply Theorem \ref{thm:maintheorem1} to obtain a corresponding weak Maz'ya-Shaposhnikova formula.

\smallskip

\noindent {\bf \flushleft Acknowledgments.} S.B. is supported by the Austrian Science Fund (FWF) projects 10.55776/F65, 10.55776/P32788 and 10.55776/ESP9, by the Belgian National Fund for Scientific Research (NFSR) project CR 40006150, and by the GNAMPA-INdAM Project 2023 ``Regolarità per problemi ellittici e parabolici con crescite non standard" (CUP\_E53C22001930001). The work of W. G\'orny was funded partially by the Austrian Science Fund (FWF), grant 10.55776/ESP88. For the purpose of open access, the authors have applied a CC BY public copyright licence to any Author Accepted Manuscript version arising from this submission.

\end{document}